\documentclass[reqno]{amsart}
\usepackage{cases}
\usepackage{amsmath}
\usepackage{pifont}
\usepackage{pifont}
\usepackage{pifont}
\usepackage{pifont}
\usepackage{pifont}
\usepackage{pifont}
\usepackage{pifont}
\usepackage{pifont}
\usepackage{pifont}
\usepackage{pifont}
\usepackage{amsfonts}
\usepackage{amsfonts}
\usepackage{amsfonts}
\usepackage{amsfonts}
\usepackage{amsmath, colortbl}
\usepackage{amsfonts}
\usepackage{amssymb}
\usepackage{mathrsfs}
\usepackage{graphicx}
\usepackage{flafter}
\newtheorem{thm}{Theorem}[section]

\newtheorem{lem}[thm]{Lemma}

\theoremstyle{definition}

\newtheorem{rem}[thm]{Remark}
\numberwithin{equation}{section}

\begin{document}
\large{
\thispagestyle{empty}

\vspace{1 true cm}

\title [Almost critical regularity conditions based on one component] {several almost critical regularity conditions based on one component of the solutions  for 3D N-S
Equations }%
\author[Daoyuan Fang  \ Chenyin Qian ]{Daoyuan Fang \hspace{0.5cm} Chenyin Qian\\
Department of Mathematics, Zhejiang University,\\ Hangzhou, 310027,
China}

\thanks{{2000 Mathematics Subject Classification.}  35Q30; \ \
76D05}%
\thanks{{Key words.} 3D Navier-Stokes equations; Leray-Hopf
weak solution; Regularity criterion}%
\thanks{This work is supported  by NSFC
11271322 and 11331005.}
\thanks{E-mail addresses: dyf@zju.edu.cn (D. Fang), qcyjcsx@163.com(C. Qian)}

\begin{abstract}

In this article, we establish several almost critical regularity conditions  such that the
weak solutions of the 3D Navier-Stokes equations become regular, based on one component of the solutions, say $u_3$ and $\partial_3u_3$.
\end{abstract}

\maketitle

\section{Introduction}
 In this paper, we consider sufficient conditions for the regularity of  weak solutions of the Cauchy problem for the
Navier-Stokes equations
\begin{equation} \label{a}
 \left\{\begin{array}{l}
\displaystyle \frac{\partial u}{\partial t}-\nu \Delta
u+(u\cdot\nabla)u+\nabla p=0,
\ \mbox{\ in}\ \mathbb{R}^{3}\times (0,T),\\
\displaystyle\nabla\cdot u=0,\hspace{3.52cm} \mbox{\ in}\ \mathbb{R}^{3}\times (0,T),\\
\displaystyle u(x, 0)=u_{0},\hspace{0.2cm}\mbox{\ in}\ \mathbb{R}^{3},\\
\end{array}
\right. \end{equation} where $u=(u_{1},u_{2},u_{3}):
\mathbb{R}^{3}\times (0,T)\rightarrow \mathbb{R}^{3}$ is the
velocity field, $ p: \mathbb{R}^{3}\times (0,T)\rightarrow
\mathbb{R}^{3}$ is a scalar pressure, and $u_{0}$ is the initial
velocity field, $\nu>0$ is the viscosity.  We set
$\nabla_{h}=(\partial_{x_{1}},\partial_{x_{2}})$ as the horizontal
gradient operator and
$\Delta_{h}=\partial_{x_{1}}^{2}+\partial_{x_{2}}^{2}$ as the
horizontal Laplacian, and $\Delta$ and $\nabla$ are the usual
Laplacian and the gradient operators, respectively.  Here we use the
classical notations
$$
(u\cdot\nabla)v=\sum_{i=1}^{3}u_{i}\partial_{x_{i}}v_{k}, \ (
k=1,2,3),\ \ \ \nabla\cdot u=\sum_{i=1}^{3}\partial_{x_{i}}u_{i},
$$
 and for sake of simplicity,  we denote $\partial_{x_{i}}$
by $\partial_{{i}}$.
\par Let us
recall the definition of Leray-Hopf weak solution. We set
$$
\mathcal {V}=\{\phi: \ \mbox{ the 3D vector valued}\ C_{0}^{\infty}
\ \mbox{functions and}\ \nabla\cdot\phi=0\},
$$
which will form the space of test functions. Let $H$ and $V$ be the
closure spaces of $\mathcal {V}$ in $L^2$ under $L^2$-topology, and
in $H^1$ under $H^1$-topology, respectively.
\par
For $u_0\in H$,  the existence of weak solutions of (1.1) was
established by Leray  \cite{[14]} and Hopf in  \cite{[9]}, that is,
$u$
satisfies the following properties:\\
(i) $u\in C_{w}([0,T); H)\cap L^{2}(0,T; V)$, and $\partial_{t}u\in
L^{1}(0,T; V^{\prime})$, where $V^{\prime}$ is the dual space of
$V$;\\
(ii) $u$ verifies (1.1) in the sense of distribution, i.e., for
every test function $\phi\in C^{\infty}([0,T);\mathcal {V})$, and
for almost every $t, t_{0}\in (0,T)$, we have
$$
\begin{array}{ll}
 \displaystyle&\displaystyle\int_{\mathbb{R}^{3}} u(x,t)\cdot\phi(x,t)dx-\int_{\mathbb{R}^{3}}
u(x,t_{0})\cdot\phi(x,t_{0})dx\vspace{2mm}\\
\displaystyle &\ \ \ \
=\displaystyle\int_{t_{0}}^{t}\int_{\mathbb{R}_{3}}[u(x,t
)\cdot(\phi_{t}(x,t)+\nu\Delta\phi(x,t))]dxds\vspace{2mm}\\
&\ \ \ \ \ \ \
+\displaystyle\int_{t_{0}}^{t}\int_{\mathbb{R}_{3}}[(u(x,t)\cdot\nabla)\phi(x,t)]\cdot
u(x,t))]dxds
\end{array}
$$
 (iii) The energy inequality,
i.e.,
$$
\|u(\cdot,t)\|_{L^{2}}^{2}+2\nu\int_{t_0}^{t}\|\nabla
u(\cdot,s)\|_{L^{2}}^{2}ds\leq\|u_{0}\|_{L^{2}}^{2},
$$
for every $t$ and almost every $t_{0}$.

It is well known, if $u_{0}\in V$, a weak solution becomes  strong
solution of (1.1) on $(0, T)$ if, in addition, it satisfies
$$
u\in C([0,T); V)\cap L^{2}(0,T; H^{2}) \ \mbox{and}\
\partial_{t}u\in L^{2}(0,T; H).
$$
We know the strong solution is regular(say, classical) and unique
(see, for example, \cite{[21]}, \cite{[22]}).
\par    For the 2D case, just as the authors  said in \cite{[3]},
 the Navier-Stokes equations \eqref{a} have unique weak and
strong solutions which exist globally in time. However, the global
regularity of solutions  for the 3D Navier-Stokes equations is a
major and challenging problem, the weak solutions are known to exist
globally in time, but the uniqueness, regularity, and continuous
dependence on initial data for weak solutions are still open
problems. Furthermore, strong solutions in the 3D case are known to
exist for a short interval of time whose length depends on the
initial data. Moreover, this strong solution is known to be unique
and to depend continuously on the initial data (see, for example,
\cite{[21]}, \cite{[22]}).
\par There are many  interesting sufficient conditions which guarantee
that a given weak solution is smooth, and the first result is
usually referred as Prodi-Serrin  (PS) conditions (see \cite{[18]} and
\cite{[20]}), i.e. if additional  the weak solution $u$ is in the
class of
\begin{equation}\label{bba} u\in L^{t}(0,T; L^{s}(\mathbb{R}^{3})),\ \
\frac{2}{t}+\frac{3}{s}=1,\ s\in]3,\infty],
\end{equation} then the
weak solution becomes regular. As to $s=3,$  Escauriaza, Seregin and
$\check{\mbox{S}}$ver$\acute{\mbox{a}}$k in \cite{[35]} established
the $L^{\infty,3}$ regularity criterion which says that if a weak
solution $u\in L^{\infty}(0,T; L^{3}(\mathbb{R}^{3}))$, then it is
regular. It is well known that if $(u, p)$ solves the Navier-Stokes
equations, then so does $(u_{\lambda}, p_{\lambda}) $ for all
$\lambda>0$ , where $u_{\lambda}(x, t)=\lambda u(\lambda x,
\lambda^2t), p_{\lambda}(x, t)=\lambda^2p(\lambda x, \lambda^2t)$.
The class of Serrin's type is important from a viewpoint of scaling
invariance, which implies that
$\|u_\lambda\|_{L^tL^s}=\|u\|_{L^{t}L^s}$ holds for all $\lambda>0$
if and only if $\frac{2}{t}+\frac{3}{s}=1$. The full regularity of
weak solutions can also be proved under alternative assumptions on
the gradient of the velocity $\nabla u$. In 1995,
Beir$\tilde{\mbox{a}}$o da Veiga \cite{[01]} established a Serrin's
type regularity criterion on the gradient of the velocity field, if
$$
\nabla u\in L^{t}(0,T; L^{s}(\mathbb{R}^{3})),\ \
\frac{2}{t}+\frac{3}{s}=2,\ s\in[\frac{3}{2},\infty].
$$
\par It is shown that the additional conditions in terms of
only one velocity component, say $u_3,$ cannot satisfy the same PS
conditions as above and have a gap remained (see, for example,
\cite{[16]}, \cite{[25]}, \cite{[3]}, \cite{[26]} and \cite{[39]}).
Similarly, when we provide sufficient conditions in terms of only
one of the nine components of the gradient of velocity field (i.e.,
the velocity Jacobian matrix), the gap seems enlarged (see, for
example, \cite{[2]}, \cite{[201]}, \cite{[202]} and \cite{[26]}).
As to the results refer to $\nabla u_3$, one can find in
\cite{[17]}, \cite{[27]},  \cite{[10]} and \cite{[39]}. The reason
to lead to  the gap is from the term $u\cdot\nabla u$.  Especially,
when we give the conditions on $u_3$ in some Lebesgue space, the
terms $\partial_iu_j$, $i,j=1,2$ are difficult to control. In order
to make sure the sufficient conditions satisfy the PS indexes,
authors may consider  the combined version of the regularity
criterion (based on two or more components of velocity or gradient
of velocity). For example, in \cite{[36]}, \cite{[37]} and
\cite{[38]}, authors investigated regularity criterion in terms of
$\partial_3 u$. For other combined version of the critical
regularity criterion, we refer to \cite{[36]}, \cite{[44]},
\cite{[41]}. There are many other versions of regularity criteria on
component of velocity, the component of gradient of velocity or the
combined of the components. For example, see \cite{[4]}, \cite{[6]},
\cite{[12]}, \cite{[13]}, \cite{[15]}, \cite{[31]}, \cite{[19]},
 \cite{[24]}.\par
 In this paper, we will get several almost critical
regularity criterion based on only one  velocity component
$u_3$ and its partial derivative $\partial_3u_3$. By using the anisotropic integrability properties on the
spaces variable, we obtain a better result than previous ones. A
crucial point is that we improve the inequality obtained in
\cite{[2]} to the anisotropic case (see Lemma \ref{l2.2} below for
detail).  In Theorem \ref{t1.1} below, we impose the assumption only
on the $\partial_3u_3$,  we see that the indexes satisfy the
``quasi-PS type" (the scaling indexes satisfy the strict inequality).
 In
Theorem \ref{t1.3}, we will give the quasi-PS type condition on $u_3$
to prove that regularity of $u$, and  we see that the coupled
condition on $\partial_3u_3$ is scaling invariant.
 \par Our main results  can be stated in the
following:
\begin{thm}\label{t1.1}
 Let $u$ be a Leray-Hopf weak solution to the 3D Navier-Stokes
equations \eqref{a} with the initial value $u_{0}\in V$. Suppose one of the following two
items are satisfied.\\
$(i)$ If $\partial_3u_3$ satisfies
\begin{eqnarray}\label{z} \sup_{0\leq t\leq T}\left\|\|
\partial_{3}u_{3}(t)\|_{L^{\alpha}_{x_3}}\right\|_{L^{\beta}_{x_1,x_2}}\leq M, \ \mbox{for some} \ M>0, \end{eqnarray}
where $\alpha$ and $\beta$ satisfy
\begin{equation} \label{f}
 1\leq\alpha\leq\beta, \ 2<\beta\leq+\infty.  \end{equation}
$(ii)$ If $u_3$ and $\partial_3u_3$ satisfy
\begin{eqnarray}\label{5ab} u_3\in L^{\infty}(0,T;
L^{3}(\mathbb{R}^3))\ \ \mbox{and}\ \ \sup_{0\leq t\leq T}\left\|\|
\partial_{3}u_{3}(t)\|_{L^{\alpha}_{x_3}}\right\|_{L^{\beta}_{x_1,x_2}}\leq M,
\ \mbox{for some} \ M>0,
\end{eqnarray}
where $\alpha$ and $\beta$ satisfy
\begin{equation} \label{f1}
 \begin{array}{l}
\displaystyle \frac{1}{\alpha}+\frac{2}{\beta}<2\ \mbox{and} \
1<\alpha\leq\beta, \ \frac{3}{2}<\beta\leq 2.
\end{array}
 \end{equation}
Then $u$ is regular.
\end{thm}
\begin{figure} \centering
\includegraphics[width=0.35\textwidth,height=0.25\textheight]{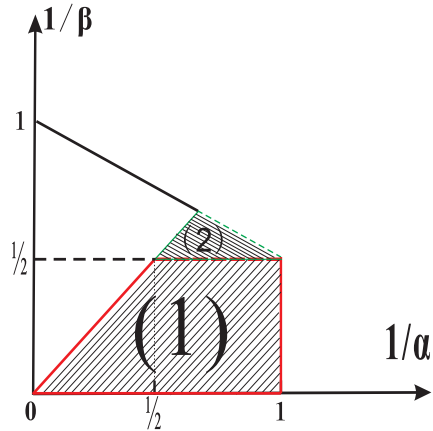}
\caption{ Range of $(\alpha,\beta)$}\label{fig:128} The domain
"\textbf{(1)}"  means the range of $(\alpha,\beta)$ in Theorem
\ref{t1.1} $(i)$. The domain "\textbf{(2)}" means the result of
 Theorem \ref{t1.1} $(ii)$.
\end{figure}
\begin{rem}\label{r1.1}
As we know that there is a large gap between the regularity criteria
which have been obtained so far only on $\partial_3u_3$ in Lebesgue
space and the PS type condition:
$$
\partial_3u_3\in L^{t}(0,T; L^{s}(\mathbb{R}^{3})),\ \
\frac{2}{t}+\frac{3}{s}=2,\ s\in[\frac{3}{2},\infty].
$$
 The purpose of $(i)$ in Theorem\ref{t1.1}  is to narrow this gap, and it
shows that our criterion is of quasi-PS type.  The range of the
$(\alpha,\beta)$ is shown by the domain ``\textbf{(1)}" in Figure 1.
The condition \eqref{z} ahows the different integrability on vertical and
horizontal components. If we choose $\alpha=1$ and $\beta$ tends to
$2^{+}$, we see that the limit is a point of the line
$1/\alpha+2/\beta=2$ (see Figure 1 for detail). When
$\alpha=\beta>2$, \eqref{z} becomes $\partial_3u_3\in
L^{\infty}(0,T; L^{\alpha}(\mathbb{R}^{3}))$, $\alpha>2$, and this
result reduce to the endpoint version of regularity criterion of
\cite{[2]} (It can be obtained by using the method of \cite{[2]}
even though the authors did not mentioned). Moreover, we recall  the
endpoint version of regularity criterion $\partial_ju_k\in
L^{\infty}(0,T; L^{3}(\mathbb{R}^{3}))$ in Theorem 1.1 of
\cite{[201]}. We see that this result is also an improvement of the
case of $j=k.$
\end{rem}
\begin{rem}\label{r1.2}
 Since the endpoint PS type condition
on $u$ makes sure the weak solution regular (see \cite{[35]}). The
$(ii)$ in Theorem\ref{t1.1} gives a depiction and comparison between
the endpoint version of regularity criterion on $u_3$ and $u$.  The
range of the $(\alpha,\beta)$ is shown by the domain ``\textbf{(2)}"
in Figure 1.  In case of  $\frac{3}{2}<\beta\leq 2$, we also see
that the line $1/\alpha+2/\beta=2$ is the limit of the case of range
of $(\alpha,\beta).$
\end{rem}

\begin{thm}\label{t1.3}
 Suppose that $u_{0}\in V$, and
$u$ is a Leray-Hopf weak solution to the 3D Navier-Stokes equations
\eqref{a}.
 Suppose that
\begin{eqnarray}\label{5a1} 1<\alpha<+\infty,\ \
\max\left\{\frac{11\alpha-12}{3(\alpha-1)},3\right\}<
s\leq\frac{11\alpha-10}{3(\alpha-1)},
\end{eqnarray}
and $u$ satisfies  the following conditions
 \begin{eqnarray}\label{5a2} u_3\in L^{\infty}(0,T;
L^{s}(\mathbb{R}^3)),
\end{eqnarray}
and
\begin{eqnarray}\label{z2}
\displaystyle\int_0^T\left\|\|
\partial_{3}u_{3}(\tau)\|_{L^{\alpha}_{x_3}}\right\|
_{L^{\beta}_{x_1,x_2}}^{p}d\tau\leq
M, \ \mbox{for some} \ M>0, \end{eqnarray} 
where
$$\displaystyle\beta=\frac{2\alpha}{(11\alpha-10)-3s(\alpha-1)}\ \mbox{and}\
p=\frac{2\alpha}{3(\alpha-1)(s-3)}.$$\\
Then $u$ is regular. 
\end{thm}
\begin{rem}\label{r1.3}  We note that $\alpha,\beta$ and $ p$ in
Theorem \ref{t1.3} satisfy $1/\alpha+2/\beta+2/p=2.$ This means
when we assume $s>3$, we can give a  scaling invariant condition on
$\partial_3u_3$. 
\end{rem}


 In following theorem, we give the assumption on $u_3$ and $\partial_{3}u_{3}$
 with the time integrability.

\begin{thm}\label{t1.4}
 Suppose that $u_{0}\in V$, and
$u$ is a Leray-Hopf weak solution to the 3D Navier-Stokes equations
\eqref{a}.
 Suppose  $u$ satisfies  the following conditions
 \begin{eqnarray}\label{5a4} \displaystyle\int_0^T\|
u_{3}(\tau)\| _{L^{s}}^{q}d\tau\leq M, \ \mbox{for some} \ M>0,
\end{eqnarray}
 and\begin{eqnarray}\label{z3}
\displaystyle\int_0^T\left\|\|
\partial_{3}u_{3}(\tau)\|_{L^{\alpha}_{x_3}}\right\|
_{L^{\beta}_{x_1,x_2}}^{p}d\tau\leq M, \ \mbox{for some} \ M>0,
\end{eqnarray} where $s$ and $q$, $\alpha$, $\beta$ and $p$ satisfy
\begin{eqnarray}\label{z4} \displaystyle
\frac{3}{s}+\frac{2}{q}<1;\ \
\frac{1}{\alpha}+\frac{2}{\beta}+\frac{2}{p}=2,
\end{eqnarray}
and
\begin{eqnarray}\label{z5} \displaystyle
\frac{3}{2}<\beta< 2, \ \frac{\beta}{2\beta-2}<\alpha\leq\beta,\
\frac{11\alpha\beta-10\beta-2\alpha}{3(\alpha-1)\beta}\leq s\leq
\infty, 
\end{eqnarray}
Then $u$ is regular.
\end{thm}

\begin{thm}\label{t1.5}
 Suppose that $u_{0}\in V$, and
$u$ is a Leray-Hopf weak solution to the 3D Navier-Stokes equations
\eqref{a}.
 Suppose  $u$ satisfies  the following conditions
 \begin{eqnarray}\label{5a5} \displaystyle\int_0^T\|
u_{3}(\tau)\| _{L^{s}}^{q}d\tau\leq M, \ \mbox{for some} \ M>0,
\end{eqnarray}
 and\begin{eqnarray}\label{z6}
\displaystyle\int_0^T\left\|\|
\partial_{3}u_{3}(\tau)\|_{L^{\alpha}_{x_3}}\right\|
_{L^{\beta}_{x_1,x_2}}^{p}d\tau\leq M, \ \mbox{for some} \ M>0,
\end{eqnarray} where $s$ and $q$, $\alpha$, $\beta$ and $p$ satisfy
\begin{eqnarray}\label{z7} \displaystyle
\frac{3}{s}+\frac{2}{q}=1;\ \
\frac{1}{\alpha}+\frac{2}{\beta}+\frac{2}{p}<2,
\end{eqnarray}
and
\begin{eqnarray}\label{z8} \displaystyle
\frac{3}{2}\leq\beta\leq 2, \ \
\frac{\beta}{2\beta-2}<\alpha\leq\beta,\ 3\leq s\leq
\frac{9\alpha\beta-6\beta-6\alpha}{(\alpha-1)\beta}, 
\end{eqnarray}
Then $u$ is regular.
\end{thm}

\begin{rem}\label{r1.4}

Very recently, J. Y. Chemin and P. Zhang in \cite{[42]} considered
the sufficient additional condition in homogeneity Sobolev spaces
rather than Lebesgue spaces, and got the regularity criterion
involving only one component of velocity
\begin{eqnarray}\label{po}
u_3 \in L^{\gamma}(0,T;
\dot{H}^{\sigma}(\mathbb{R}^3)),\
\gamma\in]4,6[.
\end{eqnarray} The derivative on $u_3$ is order of $\sigma=\frac{1}{2}+\frac{2}{\gamma}\in ]\frac 56, 1[$, and the
embedding
$\dot{H}^{\frac{1}{2}+\frac{2}{\gamma}}(\mathbb{R}^3)\hookrightarrow
L^{\eta}(\mathbb{R}^3)$ implies the range of
$\eta=\frac{3\gamma}{\gamma-2}$ is $]\frac{9}{2},6[$.  We find there
are some interesting inspirations between Theorem \ref{t1.4} or
Theorem \ref{t1.5} and \eqref{po}. In Theorem \ref{t1.4}, we see
that the condition on $\partial_3u_3$ is critical, that is scaling
invariant, the range of $p\in]4,\infty]$ is larger than the one of
$\gamma \in]4,6[$ in \eqref{po}, the condition \eqref{5a4} is weaker
than that in \eqref{po} because of $L^{\gamma}(0,T;
\dot{H}^{\sigma}(\mathbb{R}^3)) \hookrightarrow L^\gamma (0,T;
L^\eta(\mathbb{R}^3))$. We see that the indexes of $u_3$ in
\eqref{5a4} is of ``quasi-PS type", and the range of $(s, q)$ is
larger than $(\gamma,\eta)$ in \eqref{po}, and in particular, the
range of $q$ can be enlarged to $]3,\infty]$, since for any
$\epsilon>0$, there exist $\alpha, \beta$ satisfying \eqref{z5} such
that $0<\frac{11\alpha\beta-10\beta-2\alpha}{3(\alpha-1)\beta}-
3<\epsilon.$  While Theorem \ref{t1.5} gives another depiction on
the regularity criterion on $u_3$, in which we assume the indexes on
$u_3$ is critical and that on $\partial_3u_3$ is of quasi-PS type
and range of the indexes are correspondingly expanded. Moreover,
\label{r1.5} Theorem \ref{t1.4} and \ref{t1.5} also generalize the
results of \cite{[31]}. We see that the author in \cite{[31]}
considered the special case  of the  Theorem \ref{t1.4} and
\ref{t1.5} with $\alpha=\beta$.
\end{rem}

\par For the convenience, we recall the following version of
the three-dimensional  Sobolev and Ladyzhenskaya inequalities in the
whole space $\mathbb{R}^{3}$ (see, for example, \cite{[5]},
\cite{[8]}, \cite{[11]}). There exists a positive constant $C$ such
that
\begin{equation}\label{i}
\begin{array}{ll}
 \displaystyle
\|u\|_{r}&\displaystyle
\leq C \|u\|_{2}^{\frac{6-r}{2r}}\|\partial_{1}u\|_{2}^{\frac{r-2}{2r}}\|\partial_{2}u\|_{2}^{\frac{r-2}{2r}}\|\partial_{3}u\|_{2}^{\frac{r-2}{2r}}\\
&\leq C \|u\|_{2}^{\frac{6-r}{2r}}\|\nabla
u\|_{2}^{\frac{3(r-2)}{2r}},
\end{array}
\end{equation}
for every $u\in H^{1}(\mathbb{R}^{3})$ and every $r\in[2,6]$, where
$C$ is a constant depending only on $r$.
\section{{Proof of Main Results}}
\label{}In this section, under the assumptions of the Theorem
\ref{t1.1}, Theorem \ref{t1.3}, Theorem \ref{t1.4} or Theorem
\ref{t1.5} in Section 1 respectively, we prove our main results.
First of all, by using the energy inequality, for Leray-Hopf weak
solutions, we have (see, for example, \cite{[21]}, \cite{[22]} for
detail)
\begin{equation}\label{3.1}
\|u(\cdot,t)\|_{L^{2}}^{2}+2\nu\int_{0}^{t}\|\nabla
u(\cdot,s)\|_{L^{2}}^{2}ds\leq K_{1},\end{equation} for all $0<
t<T,$ where $K_{1}=\|u_{0}\|_{L^{2}}^{2}.$
\par  It is  well
known that there exists a unique strong solution  $u$ local in time
if $u_{0}\in V$. In addition, this strong solution $u\in
C([0,T^{*});V)\cap L^{2}(0,T^{*}; H^{2}(\mathbb{R}^{3}))$ is the
only weak solution with the initial datum $u_{0}$, where $(0,T^{*})$
is the maximal interval of existence of the unique strong solution.
If $T^{*}\geq T,$ then there is nothing to prove. If, on the other
hand, $T^{*}< T,$ then our strategy is to show that the $H^{1}$ norm
of this strong solution is bounded uniformly in time over the
interval $(0,T^{*})$, provided  additional conditions in Theorem
\ref{t1.1},  Theorem \ref{t1.3}, Theorem \ref{t1.4} or Theorem
\ref{t1.5} in Section 1 are valid. As a result the interval
$(0,T^{*})$ cannot be a maximal interval of existence, and
consequently $T^{*}\geq T,$ which concludes our proof. \par In order
to prove the $H^{1}$ norm of the strong solution $u$ is bounded on
interval $(0,T^{*})$, combing with the energy equality \eqref{3.1},
it is sufficient to prove
\begin{equation} \label{3.2}
\|\nabla
u\|_{L^2}^{2}+\displaystyle{\nu}\displaystyle\int_{0}^{t}\|\Delta
u\|_{L^2}^{2}d\tau\leq C,\ \forall \ t\in(0, T^{*}),
\end{equation}
where the constant $C$ depends on $T$, $K_{1} $.
 Before we prove the main theorem, we show the
following lemma.
\begin{lem}\label{l2.2}
Let us assume that
\begin{equation}\label{2}
1\leq{{\alpha}},\beta,s,a,t\leq\infty,\  \ 2<r\leq\infty, \
\mbox{and}\ 0\leq\theta\leq1,\end{equation} where $\alpha, \beta,s,
r$ and $ \theta$ satisfy
\begin{equation}\label{3}
\frac{1}{a}+\frac{1}{t}=\frac{\beta-1}{\beta},\
\end{equation}
and \begin{equation}\label{3a}
\frac{1}{(r-1)\alpha}+\frac{\theta}{\alpha}=\frac{1-\theta}{s(\alpha-1)},
\end{equation} then we have the following estimates
\begin{equation}\label{4a}
\begin{array}{ll}
 \displaystyle\left|\int_{\mathbb{R}^3}\phi f g dx_1dx_2dx_3\right|
 &\displaystyle\leq
 C\left\|\|\partial_3\phi\|_{L^\alpha_{x_3}}\right\|_{L^\beta_{x_1,x_2}}^\frac{1}{r}
 \left\|\|\partial_3\phi\|_{L^\alpha_{x_3}}\right\|_{L^{\theta(r-1)t}_{x_1,x_2}}^{\frac{\theta(r-1)}{r}}
 \left\|\|\phi\|_{L^s_{x_3}}\right\|_{L^{(1-\theta)(r-1)a}_{x_1,x_2}}^{\frac{(1-\theta)(r-1)}{r}}
 \\
 &\displaystyle\ \ \ \ \times\|f
\|_{L^2}^{\frac{r-2}{r}}\|\partial_1f
\|_{L^2}^{\frac{1}{r}}\|\partial_2f
\|_{L^2}^{\frac{1}{r}}\|g\|_{L^2}.
\end{array}
\end{equation}
\end{lem}
\begin{proof}
 Without loss of generality, we assume that the functions $\phi,f,g\in \mathcal {C}_{0}^
 {\infty}(\mathbb{R}^3)$, By using of Gagliardo-Nirenberg
and H$\ddot{\mbox{o}}$lder's inequalities,  we have
\begin{equation}\label{6}
\begin{array}{ll}
 &\displaystyle\left|\int_{\mathbb{R}^3}\phi f g dx_1dx_2dx_3\right|\\
  & \displaystyle \leq C\int_{\mathbb{R}^{2}}\left[\max_{x_{3}}|\phi|
  \left(\int_{\mathbb{R}}|f|^{2}dx_{3}\right)^{\frac{1}{2}}
  \left(\int_{\mathbb{R}}|g|^{2}dx_{3}\right)^{\frac{1}{2}} \right]dx_{1}dx_{2}\
 \displaystyle \vspace{1mm}\\
 \displaystyle & \leq\displaystyle
 C\left[\int_{\mathbb{R}^{2}}(\max_{x_{3}}|\phi|)^{r}dx_{1}dx_{2}\right]^{\frac{1}{r}}
 \left[\int_{\mathbb{R}^{2}}\left(\int_{\mathbb{R}}|f|^{2}dx_{3}\right)^{\frac{r}{r-2}}dx_{1}dx_{2}\right]^{\frac{r-2}{2r}}\\
 \displaystyle& \ \ \ \ \ \ \times \displaystyle\left[\int_{\mathbb{R}^{3}}
 |g|^{2}dx_{1}dx_{2}dx_{3}\right]^{\frac{1}{2}}
 \displaystyle \vspace{1mm}\\
 \displaystyle & \leq\displaystyle
 C\left[\int_{\mathbb{R}^{3}}|\phi|^{r-1}|\partial_{3}\phi|dx_{1}dx_{2}dx_{3}\right]^{\frac{1}{r}}
 \left[\int_{\mathbb{R}}\left(\int_{\mathbb{R}^{2}}|f|^{\frac{2r}{r-2}}dx_{1}dx_{2}\right)^{\frac{r-2}{r}}dx_{3}\right]^{\frac{1}{2}}
 \|g\|_{L^2}\vspace{2mm}\\
\displaystyle &\leq \displaystyle C
\left\|\|\partial_3\phi\|_{L^\alpha_{x_3}}\right\|_{L^\beta_{x_1,x_2}}^{\frac{1}{r}}
 \left\|\left\|\phi\right\|^{r-1}_{L^\frac{\alpha(r-1)}{\alpha-1}_{x_3}}\right\|
 _{L^{\frac{\beta}{\beta-1}}_{x_1,x_2}}^{\frac{1}{r}}
 \|f\|_{L^2}^{\frac{r-2}{r}}
 \|\partial_{1}f\|_{L^2}^{\frac{1}{r}} \|\partial_{2}f\|_{L^2}^{\frac{1}{r}}
\|g\|_{L^2}\vspace{1mm}\\
\displaystyle &\leq \displaystyle C
\left\|\|\partial_3\phi\|_{L^\alpha_{x_3}}\right\|_{L^\beta_{x_1,x_2}}^{\frac{1}{r}}
 \left\|\|\partial_3\phi\|_{L^\alpha_{x_3}}^{\theta(r-1)}
 \|\phi\|_{L^s_{x_3}}^{(1-\theta)(r-1)}\right\|_{L^{\frac{\beta}{\beta-1}}_{x_1,x_2}}^{\frac{1}{r}}
\vspace{1mm}\\
\displaystyle & \ \ \ \ \ \ \times\|f\|_{L^2}^{\frac{r-2}{r}}
 \|\partial_{1}f\|_{L^2}^{\frac{1}{r}} \|\partial_{2}f\|_{L^2}^{\frac{1}{r}}
\|g\|_{L^2}\vspace{2mm}\\
\displaystyle &\leq \displaystyle C
\left\|\|\partial_3\phi\|_{L^\alpha_{x_3}}\right\|_{L^\beta_{x_1,x_2}}^{\frac{1}{r}}
 \left\|\|\partial_3\phi\|_{L^\alpha_{x_3}}\right\|_{L^{\theta(r-1)t}_{x_1,x_2}}^{\frac{\theta(r-1)}{r}}
 \left\|\|\phi\|_{L^s_{x_3}}\right\|_{L^{(1-\theta)(r-1)a}_{x_1,x_2}}^{\frac{(1-\theta)(r-1)}{r}}
\vspace{1mm}\\
\displaystyle & \ \ \ \ \ \ \times\|f\|_{L^2}^{\frac{r-2}{r}}
 \|\partial_{1}f\|_{L^2}^{\frac{1}{r}} \|\partial_{2}f\|_{L^2}^{\frac{1}{r}}
\|g\|_{L^2}.
\end{array}
\end{equation}
in above inequalities, we  used \eqref{2} and \eqref{3}. The proof
is completed.\end{proof}

 \begin{proof}[Proof
of Theorem 1.1]
We first prove $(i)$. For the $\alpha, \beta$, we set $s=2$,
\begin{equation}\label{3.3}
r=\frac{\beta(3\alpha-2)}{\alpha(\beta+1)-\beta},\
\end{equation}
and
\begin{equation}\label{3.4}
\theta=\frac{\beta-\alpha}{2\alpha\beta-\alpha-\beta}.\
\end{equation}
then such  $s,r$ and $\theta$ satisfy \eqref{3}. We select that
\begin{equation}\label{3.5}
a=\frac{\alpha\beta+\alpha-\beta}{\alpha\beta-\beta}, \
t=\frac{\beta(\alpha\beta+\alpha-\beta)}{\beta-\alpha},
\end{equation}
then the selected $a$ and $t$ satisfy \eqref{3a}. Because of  $$
r-2=\frac{\alpha(\beta-2)}{\alpha\beta+\alpha-\beta},
$$it is easy to check that \eqref{2} is also satisfied by
\eqref{f} and \eqref{3.5}. Furthermore, we see that
\begin{equation}\label{3.6aa}
(1-\theta)(r-1)a=s=2,\ \ \theta(r-1)t=\beta.
\end{equation}
Therefore, taking the inner product of the equation \eqref{a} with
$-\Delta_{h}u$ in $L^{2}$,  we obtain
\begin{equation}\label{3.6aaa}
\begin{array}{ll}
 &\displaystyle \frac{1}{2}\frac{d}{dt}\|\nabla_{h}u\|_{L^2}^{2}+\nu\|\nabla_{h}\nabla u\|_{L^2}^{2} =\displaystyle
 \int_{\mathbb{R}^{3}}[(u\cdot \nabla)u]\Delta_{h}u dx\displaystyle \\
 \displaystyle  &\ \ \ \ \ \ \ \ \ \ \ \ \ \ = \displaystyle
\sum_{i,j=1}^{3}\int_{\mathbb{R}^{3}}u_{i}\partial_{i}u_{j}\Delta_{h}u_{j} dx\displaystyle\\
&\ \ \ \ \ \ \ \ \ \ \ \ \ \
\displaystyle=\sum_{i,j=1}^{2}\int_{\mathbb{R}^3}u_i\partial_iu_j\Delta_{h}u_jdx+
\sum_{i=1}^{2}\int_{\mathbb{R}^3}u_i\partial_iu_3\Delta_{h}
u_3dx\displaystyle\\
&\ \ \ \ \ \ \ \ \ \ \ \ \ \ \ \ \
\displaystyle+\sum_{j=1}^{2}\int_{\mathbb{R}^3}u_3\partial_3u_j\Delta_{h}u_jdx+
\int_{\mathbb{R}^3}u_3\partial_3u_3\Delta_{h}u_3dx\\
&\ \ \ \ \ \ \ \ \ \ \ \ \ \ =J_{1}(t)+J_{2}(t)+J_{3}(t)+J_4(t).
\end{array}\end{equation}
By integrating by parts a few times and using the incompressibility
condition, we get $J_{1}(t), J_{2}(t)$ as follows
$$
\begin{array}{ll}\displaystyle J_{1}(t)&\displaystyle=\frac{1}{2}\sum_{i,j=1}^{2}
\int_{\mathbb{R}^3}\partial_iu_j\partial_iu_j\partial_3u_3dx-
\int_{\mathbb{R}^3}\partial_1u_1\partial_{2}u_2\partial_3u_3dx
-\int_{\mathbb{R}^3}\partial_1u_2\partial_{2}u_1\partial_3u_3dx\end{array}$$
$$
\begin{array}{ll}\
\displaystyle
J_{2}(t)&\displaystyle=-\sum_{i,k=1}^{2}\int_{\mathbb{R}^3}\partial_ku_i\partial_iu_3\partial_ku_3dx-
\sum_{i,k=1}^{2}\int_{\mathbb{R}^3}u_i\partial_{ik}u_3\partial_ku_3dx\\
&\displaystyle=-\sum_{i,k=1}^{2}\int_{\mathbb{R}^3}\partial_ku_i\partial_iu_3\partial_ku_3dx+\frac{1}{2}
\sum_{i,k=1}^{2}\int_{\mathbb{R}^3}\partial_iu_i\partial_{k}u_3\partial_ku_3dx\\
&\displaystyle=-\sum_{i,k=1}^{2}\int_{\mathbb{R}^3}\partial_ku_i\partial_iu_3\partial_ku_3dx-\frac{1}{2}
\sum_{k=1}^{2}\int_{\mathbb{R}^3}\partial_3u_3\partial_{k}u_3\partial_ku_3dx.
\end{array}$$
From $J_{1}(t), J_{2}(t), J_{3}(t), J_{4}(t)$ it follows that
\begin{equation}\label{3.6a}
\begin{array}{ll}
 \displaystyle \frac{1}{2}\frac{d}{dt}\|\nabla_{h}u\|_{L^2}^{2}+\nu\|\nabla_{h}\nabla u\|_{L^2}^{2}
 \displaystyle \leq\displaystyle
 C\int_{\mathbb{R}^{3}}|u_{3}||\nabla u||\nabla_{h}\nabla u|dx \displaystyle
\end{array}\end{equation}
Applying Lemma \ref{l2.2} with the parameters  $r, \theta, a, t$ in
\eqref{3.3}, \eqref{3.4} and  \eqref{3.5} respectively, we have

\begin{equation}\label{3.6}
\begin{array}{ll}
 &\displaystyle \frac{1}{2}\frac{d}{dt}\|\nabla_{h}u\|_{L^2}^{2}+\nu\|\nabla_{h}\nabla u\|_{L^2}^{2}\\
 \displaystyle &\leq\displaystyle
 C\int_{\mathbb{R}^{3}}|u_{3}||\nabla u||\nabla_{h}\nabla u|dx \displaystyle \\
 \displaystyle &\leq\displaystyle
 C\left\|\|\partial_3u_3\|_{L^{\alpha}_{x_3}}\right\|_{L^{\beta}_{x_1,x_2}}^{\frac{1}{r}}
 \left\|\|\partial_3u_3\|_{L^{\alpha}_{x_3}}\right\|_{L^{\theta(r-1)t}_{x_1,x_2}}^{\frac{\theta(r-1)}{r}}
 \left\|\|u_3\|_{L^s_{x_3}}\right\|_{L^{(1-\theta)(r-1)a}_{x_1,x_2}}^{\frac{(1-\theta)(r-1)}{r}}
 \vspace{1mm}\\
 &\displaystyle\ \ \ \ \times\|\nabla u
\|_{L^2}^{\frac{r-2}{r}}\|\partial_1\nabla u
\|_{L^2}^{\frac{1}{r}}\|\partial_2\nabla u
\|_{L^2}^{\frac{1}{r}}\|\nabla_{h}\nabla u\|_{L^2}\vspace{1mm}\\
\displaystyle &\leq\displaystyle C
 \left\|\|\partial_3u_3\|_{L^{\alpha}_{x_3}}\right\|_{L^{\beta}_{x_1,x_2}}^{\frac{1+\theta(r-1)}{r}}
 \left\|u_3\right\|_{L^{2}}^{\frac{(1-\theta)(r-1)}{r}}\ \ \mbox{by}\ \eqref{3.6aa}
 \vspace{1mm}\\
 &\displaystyle\ \ \ \ \times\|\nabla u
\|_{L^2}^{\frac{r-2}{r}}\|\partial_1\nabla u
\|_{L^2}^{\frac{1}{r}}\|\partial_2\nabla u
\|_{L^2}^{\frac{1}{r}}\|\nabla_{h}\nabla u\|_{L^2}\\
\displaystyle &\leq\displaystyle C
 \left\|\|\partial_3u_3\|_{L^{\alpha}_{x_3}}\right\|_{L^{\beta}_{x_1,x_2}}^{\frac{1+\theta(r-1)}{r}}
 \left\|u_3\right\|_{L^{2}}^{\frac{(1-\theta)(r-1)}{r}}
 \|\nabla u
\|_{L^2}^{\frac{r-2}{r}}\|\nabla_{h}\nabla
u\|_{L^2}^{\frac{r+2}{r}}.
\end{array}\end{equation}
Integrating \eqref{3.6} in time, applying Young's inequality and the
energy inequality \eqref{3.1}, we get
\begin{equation}\label{3.7}
\begin{array}{ll}
 &\displaystyle \|\nabla_{h}u(t)\|_{L^2}^{2}+2\nu\|\nabla_{h}\nabla u\|_{L^2}^{2}\\
&\displaystyle\leq\|\nabla_{h}u_0\|_{L^2}^{2}
   +C\int_0^t
 \left\|\|\partial_3u_3\|_{L^{\alpha}_{x_3}}\right\|_{L^{\beta}_{x_1,x_2}}^{\frac{1+\theta(r-1)}{r}}
 \left\|u_3\right\|_{L^{2}}^{\frac{(1-\theta)(r-1)}{r}}
 \|\nabla u
\|_{L^2}^{\frac{r-2}{r}}\|\nabla_{h}\nabla
u\|_{L^2}^{\frac{r+2}{r}}d\tau\\
&\displaystyle\leq\|\nabla_{h}u_0\|_{L^2}^{2}
   +C\int_0^t
 \left\|\|\partial_3u_3\|_{L^{\alpha}_{x_3}}\right\|_{L^{\beta}_{x_1,x_2}}^{\frac{2(1+\theta(r-1))}{r-2}}
 \|\nabla u
\|_{L^2}^{2}d\tau+\nu\int_0^t\|\nabla_{h}\nabla u\|_{L^2}^{2}d\tau.
\end{array}\end{equation}
Absorbing the last term in \eqref{3.7}, and using \eqref{z} and
\eqref{3.1}, we have

\begin{equation}\label{3.8}
\begin{array}{ll}
 &\displaystyle \|\nabla_{h}u(t)\|_{L^2}^{2}+\nu\|\nabla_{h}\nabla u\|_{L^2}^{2}
 \displaystyle\leq C,
\end{array}\end{equation}
where the constant $C$ depends only on $M,K_1$. Next, we  also use
$-\Delta u$ as test function, and get
$$\begin{array}{ll} &\displaystyle \frac{1}{2}\frac{d}{dt}\|\nabla
u\|_{L^2}^{2}+\nu\|\Delta u\|_{L^2}^{2}\\&=\displaystyle
\sum_{i,j,k=1}^{3}\int_{\mathbb{R}^{3}}u_{i}\partial_{i}u_{j}\partial_{kk}u_{j} dx\displaystyle\\
&\displaystyle=\sum_{j=1}^{3}\int_{\mathbb{R}^3}u_3\partial_3u_j\Delta_{h}u_jdx+
\sum_{i=1}^{2}\sum_{j=1}^{3}\int_{\mathbb{R}^3}u_i\partial_iu_j\Delta
u_jdx+\sum_{j=1}^{3}\int_{\mathbb{R}^3}u_3\partial_3u_j\partial_{33}u_jdx\\
&=L_{1}(t)+L_{2}(t)+L_{3}(t)
\end{array}$$
By integrating by parts a few times and using the incompressibility
condition, we get $L_{1}(t), L_{2}(t), L_{3}(t)$ as follows
$$
\begin{array}{ll}\displaystyle L_{1}(t)&\displaystyle=-\sum_{j=1}^{3}\sum_{k=1}^{2}\int_{\mathbb{R}^3}\partial_ku_3\partial_3u_j\partial_ku_jdx-
\sum_{j=1}^{3}\sum_{k=1}^{2}\int_{\mathbb{R}^3}u_3\partial_{3k}u_j\partial_ku_jdx\\
&\displaystyle=-\sum_{j=1}^{3}\sum_{k=1}^{2}\int_{\mathbb{R}^3}\partial_ku_3\partial_3u_j\partial_ku_jdx+\frac{1}{2}
\sum_{j=1}^{3}\sum_{k=1}^{2}\int_{\mathbb{R}^3}\partial_3u_3\partial_{k}u_j\partial_ku_jdx,
\end{array}$$
$$
\begin{array}{ll}\
\displaystyle
L_{2}(t)&\displaystyle=-\sum_{i=1}^{2}\sum_{j=1}^{3}\sum_{k=1}^{3}\int_{\mathbb{R}^3}\partial_ku_i\partial_iu_j\partial_ku_jdx-
\sum_{i=1}^{2}\sum_{j=1}^{3}\sum_{k=1}^{3}\int_{\mathbb{R}^3}u_i\partial_{ik}u_j\partial_ku_jdx\\
&=\displaystyle-\sum_{i=1}^{2}\sum_{j=1}^{3}\sum_{k=1}^{3}\int_{\mathbb{R}^3}\partial_ku_i\partial_iu_j\partial_ku_jdx+\frac{1}{2}
\sum_{i=1}^{2}\sum_{j=1}^{3}\sum_{k=1}^{3}\int_{\mathbb{R}^3}\partial_iu_i\partial_{k}u_j\partial_ku_jdx,
\end{array}$$
$$
\begin{array}{ll}\
\displaystyle
L_{3}(t)&\displaystyle=-\frac{1}{2}\sum_{j=1}^{3}\int_{\mathbb{R}^3}\partial_3u_3\partial_3u_j\partial_3u_jdx=
\frac{1}{2}\sum_{j=1}^{3}\int_{\mathbb{R}^3}(\partial_1u_1+\partial_2u_2)\partial_3u_j\partial_3u_jdx.
\end{array}$$
Therefore,  by \eqref{i} and H$\ddot{\mbox{o}}$lder's inequalities,
for every $i\ (i=1,2,3)$ we have
\begin{equation} \label{3.10}\begin{array}{ll}\displaystyle
|L_{i}(t)|& \displaystyle\leq C\int_{\mathbb{R}^3}|\nabla_h u| |\nabla u|^{2}dx\vspace{1mm}\\
& \displaystyle\leq C\|\nabla_{h}u\|_{L^2}\|\nabla u\|_{L^4}^{2}\vspace{1mm}\\
&\displaystyle\leq C\|\nabla_{h}u\|_{L^2}\|\nabla
u\|_{L^2}^{\frac{1}{2}}\|\nabla_{h}\nabla u\|_{L^2}\|\Delta
u\|_{L^2}^{\frac{1}{2}},
\end{array}\end{equation}
and hance we have
\begin{equation} \label{3.11}\begin{array}{ll}\displaystyle \frac{1}{2}\frac{d}{dt}\|\nabla
u\|_{L^2}^{2}+\nu\|\Delta u\|_{L^2}^{2}\leq \displaystyle
C\|\nabla_{h}u\|_{L^2}\|\nabla
u\|_{L^2}^{\frac{1}{2}}\|\nabla_{h}\nabla u\|_{L^2}\|\Delta
u\|_{L^2}^{\frac{1}{2}}.
\end{array}\end{equation}
Integrating \eqref{3.11}, applying  H$\ddot{\mbox{o}}$lder's and
Young's inequalities and combining \eqref{3.8} and \eqref{3.1}, we
obtain
\begin{equation} \label{3.12}\begin{array}{ll}
&\|\nabla
u\|_{L^2}^{2}+\displaystyle2\nu\displaystyle\int_{0}^{t}\|\Delta
u\|_{L^2}^{2}d\tau\\
\displaystyle &\hspace{0.3cm} \leq\|\nabla
u(0)\|_{L^2}^{2}+\left(\sup_{0\leq s\leq t}\|\nabla_{h}
u\|_{L^2}\right)\left(\displaystyle\int_{0}^{t}\|\nabla
u\|_{L^2}^{2}d\tau \right)^{\frac{1}{4}}\\
&\ \ \ \ \ \times\left(\displaystyle\int_{0}^{t}\|\nabla_{h}\nabla
u\|_{L^2}^{2}d\tau\right)^{\frac{1}{2}}\left(\displaystyle\int_{0}^{t}\|\Delta
u\|_{L^2}^{2}d\tau\right)^{\frac{1}{4}}\\
\displaystyle &\hspace{0.3cm} \leq\|\nabla
u(0)\|_{L^2}^{2}+C\displaystyle
 \left(\int_{0}^{t}\|\Delta
 u\|_{L^2}^{2}d\tau\right)^{{\frac{1}{4}}}\\
\displaystyle &\hspace{0.3cm} \leq\|\nabla
u(0)\|_{L^2}^{2}+C\displaystyle+
 \nu\int_{0}^{t}\|\Delta
 u\|_{L^2}^{2}d\tau
\end{array}\end{equation}
Absorbing the last term on the right hand side of \eqref{3.12}, it
immediately implies that \eqref{3.2}. We complete the proof of
Theorem \ref{t1.1} $(i)$.
 \par Next, we prove  $(ii)$. Similar to the proof of Theorem \ref{t1.1}  $(i)$,
we will apply Lemma \ref{l2.2} get the desired result. Therefore,
for the $\alpha,\beta$ in \eqref{5ab} and \eqref{f1}, we set
 $s=3$,
\begin{equation}\label{3.13}
r=\frac{\beta(4\alpha-3)}{\alpha(\beta+1)-\beta},\
\end{equation}
and
\begin{equation}\label{3.14}
\theta=\frac{\beta-\alpha}{3\alpha\beta-\alpha-2\beta}.\
\end{equation}
then such  $s,r$ and $\theta$ satisfy \eqref{3}. We select that
\begin{equation}\label{3.15}
a=\frac{\alpha\beta+\alpha-\beta}{\alpha\beta-\beta}, \
t=\frac{\beta(\alpha\beta+\alpha-\beta)}{\beta-\alpha},
\end{equation}
then the selected $a$ and $t$ satisfy \eqref{3a}. Note that$$
r-2=\frac{2\alpha\beta-\beta-2\alpha}{\alpha\beta+\alpha-\beta},
$$ by \eqref{f1}, we see that \eqref{3.15} and above equality imply \eqref{2} holds, and
furthermore,
$$
(1-\theta)(r-1)a=s=3,\ \ \theta(r-1)t=\beta.
$$
Therefore, taking the inner product of the equation \eqref{a} with
$-\Delta_{h}u$ in $L^{2}$, applying Lemma \ref{l2.2} with the
parameters in \eqref{3.13}-\eqref{3.15}, similar to the proof of
Theorem \ref{t1.1}, we have
\begin{equation}\label{3.16}
\begin{array}{ll}
 &\displaystyle \frac{1}{2}\frac{d}{dt}\|\nabla_{h}u\|_{L^2}^{2}+\nu\|\nabla_{h}\nabla u\|_{L^2}^{2}\\
 \displaystyle &\leq\displaystyle
 C\int_{\mathbb{R}^{3}}|u_{3}||\nabla u||\nabla_{h}\nabla u|dx \displaystyle \\
 \displaystyle &\leq\displaystyle
 C\left\|\|\partial_3u_3\|_{L^{\alpha}_{x_3}}\right\|_{L^{\beta}_{x_1,x_2}}^{\frac{1}{r}}
 \left\|\|\partial_3u_3\|_{L^{\alpha}_{x_3}}\right\|_{L^{\theta(r-1)t}_{x_1,x_2}}^{\frac{\theta(r-1)}{r}}
 \left\|\|u_3\|_{L^s_{x_3}}\right\|_{L^{(1-\theta)(r-1)a}_{x_1,x_2}}^{\frac{(1-\theta)(r-1)}{r}}
 \vspace{1mm}\\
 &\displaystyle\ \ \ \ \times\|\nabla u
\|_{L^2}^{\frac{r-2}{r}}\|\partial_1\nabla u
\|_{L^2}^{\frac{1}{r}}\|\partial_2\nabla u
\|_{L^2}^{\frac{1}{r}}\|\nabla_{h}\nabla u\|_{L^2}\vspace{1mm}\\
\displaystyle &\leq\displaystyle C
 \left\|\|\partial_3u_3\|_{L^{\alpha}_{x_3}}\right\|_{L^{\beta}_{x_1,x_2}}^{\frac{1+\theta(r-1)}{r}}
 \left\|u_3\right\|_{L^{3}}^{\frac{(1-\theta)(r-1)}{r}}
 \vspace{1mm}\\
 &\displaystyle\ \ \ \ \times\|\nabla u
\|_{L^2}^{\frac{r-2}{r}}\|\partial_1\nabla u
\|_{L^2}^{\frac{1}{r}}\|\partial_2\nabla u
\|_{L^2}^{\frac{1}{r}}\|\nabla_{h}\nabla u\|_{L^2}\\
\displaystyle &\leq\displaystyle C
 \left\|\|\partial_3u_3\|_{L^{\alpha}_{x_3}}\right\|_{L^{\beta}_{x_1,x_2}}^{\frac{1+\theta(r-1)}{r}}
 \left\|u_3\right\|_{L^{3}}^{\frac{(1-\theta)(r-1)}{r}}
 \|\nabla u
\|_{L^2}^{\frac{r-2}{r}}\|\nabla_{h}\nabla
u\|_{L^2}^{\frac{r+2}{r}}.
\end{array}\end{equation}
Integrating \eqref{3.16} in time, applying Young's inequality and
the assumption \eqref{5ab}, we get
\begin{equation}\label{3.17}
\begin{array}{ll}
 &\displaystyle \|\nabla_{h}u(t)\|_{L^2}^{2}+2\nu\|\nabla_{h}\nabla u\|_{L^2}^{2}\\
&\displaystyle\leq\|\nabla_{h}u_0\|_{L^2}^{2}
   +C\int_0^t
 \left\|\|\partial_3u_3\|_{L^{\alpha}_{x_3}}\right\|_{L^{\beta}_{x_1,x_2}}^{\frac{1+\theta(r-1)}{r}}
 \left\|u_3\right\|_{L^{3}}^{\frac{(1-\theta)(r-1)}{r}}
 \|\nabla u
\|_{L^2}^{\frac{r-2}{r}}\|\nabla_{h}\nabla
u\|_{L^2}^{\frac{r+2}{r}}d\tau\\
&\displaystyle\leq\|\nabla_{h}u_0\|_{L^2}^{2}
   +C\int_0^t
 \left\|\|\partial_3u_3\|_{L^{\alpha}_{x_3}}\right\|_{L^{\beta}_{x_1,x_2}}^{\frac{2(1+\theta(r-1))}{r-2}}
 \|\nabla u
\|_{L^2}^{2}d\tau+\nu\int_0^t\|\nabla_{h}\nabla u\|_{L^2}^{2}d\tau.
\end{array}\end{equation}
Absorbing the last term in \eqref{3.17}, and using \eqref{5ab} and
the energy inequality \eqref{3.1}, we have
\begin{equation}\label{3.18}
\begin{array}{ll}
 &\displaystyle \|\nabla_{h}u(t)\|_{L^2}^{2}+\nu\|\nabla_{h}\nabla u\|_{L^2}^{2}
 \displaystyle\leq C,
\end{array}\end{equation}
where the constant $C$ depends only on $M,K_1$. For the rest, we can
give the same process as in Theorem \ref{t1.1} $(i)$ to prove
\eqref{3.2}, and then to complete the proof of this theorem.
\end{proof}
\begin{proof}[Proof
of Theorem \ref{t1.3}] In view of the condition
$$
1<\alpha<+\infty,\ \
\max\left\{\frac{11\alpha-12}{3(\alpha-1)},3\right\}<
s\leq\frac{11\alpha-10}{3(\alpha-1)},
$$
we give the parameters $\beta,\theta,a, t, r$ as follows, and we
will check  one by one that all of them satisfy the assumptions in
Lemma \ref{l2.2}.
\begin{equation}\label{3.19}
\beta=\frac{2\alpha}{(11\alpha-10)-3s(\alpha-1)},\
\end{equation}
\begin{equation}\label{3.20}
r=\frac{2s(\alpha-1)+2\alpha}{(13\alpha-12)-3s(\alpha-1)},\
\end{equation}
\begin{equation}\label{3.21}
\theta=\frac{3s(\alpha-1)-11\alpha+12}{5s(\alpha-1)-11\alpha+12},\
\end{equation}
\begin{equation}\label{3.22}
a=\frac{(13\alpha-12)-3s(\alpha-1)}{2(\alpha-1)},\
\end{equation}
and
\begin{equation}\label{3.23}
t=\frac{2\alpha[(13\alpha-12)-3s(\alpha-1)]}{\left[(11\alpha-10)-3s(\alpha-1)\right][
3s(\alpha-1)-11\alpha+12]}.\
\end{equation}
By\eqref{5a1}, we have $\beta>1$. In fact,
$s\leq\frac{11\alpha-10}{3(\alpha-1)}$ (if
$s=\frac{11\alpha-10}{3(\alpha-1)}$, then $\beta=\infty$) implies
that
$$
s>\frac{9\alpha-10}{3(\alpha-1)}\Longleftrightarrow\beta>1,
$$
and then also by \eqref{5a1}, we see that
$$
s>\frac{11\alpha-12}{3(\alpha-1)}\Longrightarrow
s>\frac{9\alpha-10}{3(\alpha-1)}.
$$
For $r$, because of \begin{equation}\label{3.24}
s\leq\frac{11\alpha-10}{3(\alpha-1)}\Longrightarrow
s<\frac{13\alpha-12}{3(\alpha-1)},
\end{equation}
by \eqref{3.20} and \eqref{3.24}, we have
\begin{equation}\label{3.25}
r-2=\frac{8(s-3)(\alpha-1)}{(13\alpha-12)-3s(\alpha-1)}>0\Longleftarrow
\left\{\begin{array}{l} \displaystyle s>3\\
\displaystyle \alpha>1.\\
\end{array}
\right.
\end{equation}
By\eqref{5a2} and \eqref{3.21}, it is obviously that
$0\leq\theta<1.$ As to $a$, we see that
$$
s\leq\frac{11\alpha-10}{3(\alpha-1)}\Longrightarrow a\geq1,
$$
and moreover, we have
$$
t-1=\frac{9(\alpha-1)^2s^2-6(12\alpha-11)(\alpha-1)s+(21\alpha-20)(7\alpha-6)
}{\left[(11\alpha-10)-3s(\alpha-1)\right][
3s(\alpha-1)-11\alpha+12]}>0.
$$
Besides, one can check that
\begin{equation} \label{3.26}
 \left\{\begin{array}{l}
\displaystyle \theta(r-1)t=\beta,\vspace{1mm}\\
\displaystyle (1-\theta)(r-1)a=s,\vspace{1mm}\\
\displaystyle \frac{1}{a}+\frac{1}{t}=\frac{\beta-1}{\beta},\vspace{1mm}\\
\displaystyle
\frac{1-\theta}{s}+\frac{\theta(1-\alpha)}{\alpha}=\frac{\alpha-1}{\alpha(r-1)}.
\end{array}
\right. \end{equation} Therefore, all the conditions of Lemma
\ref{l2.2} are satisfied. Similar to Theorem \ref{t1.1}, we begin
with \eqref{3.6} and by using the parameters defined in
\eqref{3.19}-\eqref{3.23} and Lemma \ref{l2.2}, we get
\begin{equation}\label{3.27}
\begin{array}{ll}
 \displaystyle&\displaystyle \frac{1}{2}\frac{d}{dt}\|\nabla_{h}u\|_{L^2}^{2}+\nu\|\nabla_{h}\nabla u\|_{L^2}^{2}\\
 \displaystyle &\leq\displaystyle
 C\int_{\mathbb{R}^{3}}|u_{3}||\nabla u||\nabla_{h}\nabla u|dx \displaystyle \\
 \displaystyle &\leq\displaystyle
 C\left\|\|\partial_3u_3\|_{L^{\alpha}_{x_3}}\right\|_{L^{\beta}_{x_1,x_2}}^{\frac{1}{r}}
 \left\|\|\partial_3u_3\|_{L^{\alpha}_{x_3}}\right\|_{L^{\theta(r-1)t}_{x_1,x_2}}^{\frac{\theta(r-1)}{r}}
 \left\|\|u_3\|_{L^s_{x_3}}\right\|_{L^{(1-\theta)(r-1)a}_{x_1,x_2}}^{\frac{(1-\theta)(r-1)}{r}}
 \vspace{1mm}\\
 &\displaystyle\ \ \ \ \times\|\nabla u
\|_{L^2}^{\frac{r-2}{r}}\|\partial_1\nabla u
\|_{L^2}^{\frac{1}{r}}\|\partial_2\nabla u
\|_{L^2}^{\frac{1}{r}}\|\nabla_{h}\nabla u\|_{L^2}\vspace{1mm}\\
\displaystyle &\leq\displaystyle C
 \left\|\|\partial_3u_3\|_{L^{\alpha}_{x_3}}\right\|_{L^{\beta}_{x_1,x_2}}^{\frac{1+\theta(r-1)}{r}}
 \left\|u_3\right\|_{L^{s}}^{\frac{(1-\theta)(r-1)}{r}}
 \vspace{1mm}\\
 &\displaystyle\ \ \ \ \times\|\nabla u
\|_{L^2}^{\frac{r-2}{r}}\|\partial_1\nabla u
\|_{L^2}^{\frac{1}{r}}\|\partial_2\nabla u
\|_{L^2}^{\frac{1}{r}}\|\nabla_{h}\nabla u\|_{L^2}\\
\displaystyle &\leq\displaystyle C
 \left\|\|\partial_3u_3\|_{L^{\alpha}_{x_3}}\right\|_{L^{\beta}_{x_1,x_2}}^{\frac{1+\theta(r-1)}{r}}
 \left\|u_3\right\|_{L^{s}}^{\frac{(1-\theta)(r-1)}{r}}
 \|\nabla u
\|_{L^2}^{\frac{r-2}{r}}\|\nabla_{h}\nabla
u\|_{L^2}^{\frac{r+2}{r}}.
\end{array}\end{equation}
Integrating \eqref{3.26} in time, applying Young's inequality and
the assumption \eqref{5a2}, we get
\begin{equation}\label{3.28}
\begin{array}{ll}
 &\displaystyle \|\nabla_{h}u(t)\|_{L^2}^{2}+2\nu\|\nabla_{h}\nabla u\|_{L^2}^{2}\\
&\displaystyle\leq\|\nabla_{h}u_0\|_{L^2}^{2}
   +C\int_0^t
 \left\|\|\partial_3u_3\|_{L^{\alpha}_{x_3}}\right\|_{L^{\beta}_{x_1,x_2}}^{\frac{1+\theta(r-1)}{r}}
 \left\|u_3\right\|_{L^{s}}^{\frac{(1-\theta)(r-1)}{r}}
 \|\nabla u
\|_{L^2}^{\frac{r-2}{r}}\|\nabla_{h}\nabla
u\|_{L^2}^{\frac{r+2}{r}}d\tau\\
&\displaystyle\leq\|\nabla_{h}u_0\|_{L^2}^{2}
   +C\int_0^t
 \left\|\|\partial_3u_3\|_{L^{\alpha}_{x_3}}\right\|_{L^{\beta}_{x_1,x_2}}^{\frac{2(1+\theta(r-1))}{r-2}}
 \|\nabla u
\|_{L^2}^{2}d\tau+\nu\int_0^t\|\nabla_{h}\nabla u\|_{L^2}^{2}d\tau.
\end{array}\end{equation}
Absorbing the last term in \eqref{3.28}, we have
\begin{equation}\label{3.29}
\begin{array}{ll}
\displaystyle &\displaystyle \|\nabla_{h}u(t)\|_{L^2}^{2}+\nu\|\nabla_{h}\nabla u\|_{L^2}^{2}\\
&\displaystyle\ \ \ \ \ \leq\|\nabla_{h}u_0\|_{L^2}^{2}
   +C\int_0^t
 \left\|\|\partial_3u_3\|_{L^{\alpha}_{x_3}}\right\|_{L^{\beta}_{x_1,x_2}}^{\frac{2(1+\theta(r-1))}{r-2}}
 \|\nabla u
\|_{L^2}^{2}d\tau\\
&\displaystyle\ \ \ \ \ =\|\nabla_{h}u_0\|_{L^2}^{2}
   +C\int_0^t
 \left\|\|\partial_3u_3\|_{L^{\alpha}_{x_3}}\right\|_{L^{\beta}_{x_1,x_2}}^{\frac{\alpha}{2(s-3)(\alpha-1)}}
 \|\nabla u
\|_{L^2}^{2}d\tau,
\end{array}\end{equation}
where we note that
$\frac{2(1+\theta(r-1))}{r-2}=\frac{\alpha}{2(s-3)(\alpha-1)}.$
Next, we apply the estimates on $\|\nabla u(t)\|_{L^2}^{2}$. In view
of \eqref{3.11}, and integrating it in time, applying
H$\ddot{\mbox{o}}$lder's and Young's inequalities and combining
\eqref{3.29} and \eqref{3.1}, we obtain
\begin{equation} \label{3.30}\begin{array}{ll}
&\|\nabla
u\|_{L^2}^{2}+\displaystyle2\nu\displaystyle\int_{0}^{t}\|\Delta
u\|_{L^2}^{2}d\tau\\
\displaystyle &\hspace{0.3cm} \leq\|\nabla
u_0\|_{L^2}^{2}+\left(\sup_{0\leq s\leq t}\|\nabla_{h}
u\|_{L^2}\right)\left(\displaystyle\int_{0}^{t}\|\nabla
u\|_{L^2}^{2}d\tau \right)^{\frac{1}{4}}\\
&\ \ \ \ \ \times\left(\displaystyle\int_{0}^{t}\|\nabla_{h}\nabla
u\|_{L^2}^{2}d\tau\right)^{\frac{1}{2}}\left(\displaystyle\int_{0}^{t}\|\Delta
u\|_{L^2}^{2}d\tau\right)^{\frac{1}{4}}\\
\displaystyle &\hspace{0.3cm} \leq C\displaystyle\left(\int_0^t
 \left\|\|\partial_3u_3\|_{L^{\alpha}_{x_3}}\right\|_{L^{\beta}_{x_1,x_2}}^{\frac{\alpha}{2(s-3)(\alpha-1)}}
 \|\nabla u
\|_{L^2}^{2}d\tau\right)\times
 \left(\int_{0}^{t}\|\Delta
 u\|_{L^2}^{2}d\tau\right)^{{\frac{1}{4}}}\\
 &\hspace{0.5cm} \ \ + \displaystyle\|\nabla_{h}
u_0\|_{L^2}^{2}\left(\displaystyle\int_{0}^{t}\|\Delta
u\|_{L^2}^{2}d\tau\right)^{\frac{1}{4}}+\|\nabla u_0\|_{L^2}^{2}.
\end{array}\end{equation}
By H$\ddot{\mbox{o}}$lder's and Young's inequalities, one has
\begin{equation} \label{3.31}\begin{array}{ll}
&\|\nabla
u\|_{L^2}^{2}+\displaystyle2\nu\displaystyle\int_{0}^{t}\|\Delta
u\|_{L^2}^{2}d\tau\\
\displaystyle &\hspace{0.3cm} \leq C\left(1+\|\nabla
u_0\|_{L^2}^{\frac{8}{3}}\right)+C\displaystyle\left(\int_0^t
 \left\|\|\partial_3u_3\|_{L^{\alpha}_{x_3}}\right\|_{L^{\beta}_{x_1,x_2}}^{\frac{\alpha}{2(s-3)(\alpha-1)}}
 \|\nabla u
\|_{L^2}^{2}d\tau\right)^{\frac{4}{3}}\\
 &\hspace{0.5cm} \ \ \displaystyle +\nu\int_{0}^{t}\|\Delta
 u\|_{L^2}^{2}d\tau\\
\displaystyle &\hspace{0.3cm} \leq C\displaystyle\left(\int_0^t
 \left\|\|\partial_3u_3\|_{L^{\alpha}_{x_3}}\right\|_{L^{\beta}_{x_1,x_2}}^{\frac{2\alpha}{3(s-3)(\alpha-1)}}
 \|\nabla u
\|_{L^2}^{2}d\tau\right)\times
 \left(\int_{0}^{t}\|\nabla
 u\|_{L^2}^{2}d\tau\right)^{{\frac{1}{4}}}\\
 &\hspace{0.5cm} \ \ + \displaystyle C\left(1+\|\nabla
u_0\|_{L^2}^{\frac{8}{3}}\right)+\nu\int_{0}^{t}\|\Delta
 u\|_{L^2}^{2}d\tau.
\end{array}\end{equation}
Absorbing the last term on the right hand side of \eqref{3.31}, and
thanks to the energy inequality \eqref{3.1}, we get
\begin{equation} \label{3.32}\begin{array}{ll}
&\|\nabla
u\|_{L^2}^{2}+\displaystyle\nu\displaystyle\int_{0}^{t}\|\Delta
u\|_{L^2}^{2}d\tau\\
\displaystyle &\hspace{0.3cm} \leq C\left(1+\|\nabla
u_0\|_{L^2}^{\frac{8}{3}}\right)+C\left(\displaystyle\int_0^t
 \left\|\|\partial_3u_3\|_{L^{\alpha}_{x_3}}\right\|_{L^{\beta}_{x_1,x_2}}^{\frac{2\alpha}{3(s-3)(\alpha-1)}}
 \|\nabla u
\|_{L^2}^{2}d\tau\right).
\end{array}\end{equation}
Therefore, by Gronwall's inequality and \eqref{z2}, we obtain
\begin{equation} \label{3.33}\begin{array}{ll}
&\|\nabla
u\|_{L^2}^{2}+\displaystyle\nu\displaystyle\int_{0}^{t}\|\Delta
u\|_{L^2}^{2}d\tau\leq C\left(1+\|\nabla
u_0\|_{L^2}^{\frac{8}{3}}\right)(1+M)e^{CM},
\end{array}\end{equation}
 for all  $t\in(0,T^{*})$. Therefore, the $ H^1$ norm of the strong
solution $u$ is bounded on the maximal interval of existence $(0,
T^{*})$. This completes the proof of Theorem \ref{t1.3}.
\end{proof}

\begin{proof}[Proof
of Theorem \ref{t1.4}] Since $\alpha, \beta$ and $ s$ satisfy
\eqref{z4} and \eqref{z5}, for any arbitrary small positive constant
$\epsilon$ satisfying $0<\epsilon <\min\left\{\frac{4}{10},
\frac{8\beta-12}{11\beta-12}\right\}$,
 we can choose $\alpha$ such that
\begin{equation}\label{3.41}
\frac{\beta}{2\beta-2}<\alpha\leq
\frac{(4-10\epsilon)\beta}{(8-11\epsilon)\beta+2\epsilon-8},
\end{equation}
where we used $\beta<2$, and then we choose
\begin{equation}\label{3.41a}
s=\frac{\alpha\beta-2\beta+2\alpha}{(1-\epsilon)(\alpha-1)\beta}.
\end{equation}
From \eqref{3.41}, it is easy to check that
$$
s\geq\frac{11\alpha\beta-10\beta-2\alpha}{3(\alpha-1)\beta}>3,
$$
where we use the fact that
$\frac{\beta}{2\beta-2}<\alpha\Longrightarrow
\frac{11\alpha\beta-10\beta-2\alpha}{3(\alpha-1)\beta}>3.$ Next, we
set
\begin{equation}\label{3.41aa}
r=\frac{(s\alpha+\alpha-s)\beta}{\alpha\beta+\alpha-\beta},\
\end{equation}
\begin{equation}\label{3.42}
\theta=\frac{\beta-\alpha}{s\alpha\beta-s\beta-\alpha+\beta},\
\end{equation}
\begin{equation}\label{3.43}
a=\frac{\alpha\beta+\alpha-\beta}{(\alpha-1)\beta},\
\end{equation}
\begin{equation}\label{3.44}
t=\frac{(\alpha\beta+\alpha-\beta)\beta}{\beta-\alpha}.
\end{equation}
 From \eqref{3.41aa}, we have
$$
r-2=\frac{(s-1)\alpha\beta-(s-2)\beta-2\alpha}{\alpha\beta+\alpha-\beta},
$$
and by \eqref{z5}, we have $3\alpha\beta s-11\alpha\beta-3\beta
s+10\beta+2\alpha\geq0$ and $2\alpha\beta -\beta-2\alpha>0.$
Therefore, one has
$$
3\alpha\beta s-11\alpha\beta-3\beta
s+10\beta+2\alpha>0\Longleftrightarrow3[(s-1)\alpha\beta-(s-2)\beta-2\alpha]\geq4(2\alpha\beta
-\beta-2\alpha)> 0,
$$
 and finally we get $r>2 $ and $(s-1)\alpha\beta-(s-2)\beta-2\alpha>0$. Since
$$
(s-1)\alpha\beta>(s-2)\beta+2\alpha\Longleftrightarrow
s\alpha\beta-s\beta-\alpha+\beta>\alpha\beta+\alpha-\beta=(\alpha-1)\beta+\alpha>0,
$$
it is easy for us to get $0\leq\theta<1.$ Moreover,  $a\geq 1$ and
$t\geq 1$ are obviously.  All parameters selected above satisfy the
conditions of Lemma \ref{l2.2}, and similar to Theorem \ref{t1.3}
(see \eqref{3.26} and \eqref{3.27}), one has
\begin{equation}\label{3.45}
\begin{array}{ll}
 &\displaystyle \|\nabla_{h}u(t)\|_{L^2}^{2}+\nu\|\nabla_{h}\nabla u\|_{L^2}^{2}\\
&\displaystyle\leq\|\nabla_{h}u_0\|_{L^2}^{2}
   +C\int_0^t
 \left\|\|\partial_3u_3\|_{L^{\alpha}_{x_3}}\right\|_{L^{\beta}_{x_1,x_2}}^{\frac{1+\theta(r-1)}{r}}
 \left\|u_3\right\|_{L^{s}}^{\frac{(1-\theta)(r-1)}{r}}
 \|\nabla u
\|_{L^2}^{\frac{r-2}{r}}\|\nabla_{h}\nabla
u\|_{L^2}^{\frac{r+2}{r}}d\tau\\
&\displaystyle\leq\|\nabla_{h}u_0\|_{L^2}^{2}
   +C\int_0^t
 \left\|\|\partial_3u_3\|_{L^{\alpha}_{x_3}}\right\|_{L^{\beta}_{x_1,x_2}}^{\frac{2(1+\theta(r-1))}{r-2}}
 \left\|u_3\right\|_{L^{s}}^{\frac{2(1-\theta)(r-1)}{r-2}}\|\nabla u
\|_{L^2}^{2}d\tau\\
&\displaystyle=\|\nabla_{h}u_0\|_{L^2}^{2}
   +C\int_0^t
 \left\|\|\partial_3u_3\|_{L^{\alpha}_{x_3}}\right\|_{L^{\beta}_{x_1,x_2}}
 ^{\frac{2\alpha\beta}{(s-1)\alpha\beta-(s-2)\beta-2\alpha}}
 \left\|u_3\right\|_{L^{s}}^{\frac{2(\alpha-1)\beta s}{(s-1)\alpha\beta-(s-2)\beta-2\alpha}}\|\nabla u
\|_{L^2}^{2}d\tau.
\end{array}\end{equation}
Next, we apply the estimates on $\|\nabla u(t)\|_{L^2}^{2}$. In view
of \eqref{3.11}, and integrating it in time, applying
H$\ddot{\mbox{o}}$lder's and Young's inequalities and combining
\eqref{3.45} and \eqref{3.1}, we obtain
\begin{equation} \label{3.46}\begin{array}{ll}
&\|\nabla
u\|_{L^2}^{2}+\displaystyle2\nu\displaystyle\int_{0}^{t}\|\Delta
u\|_{L^2}^{2}d\tau\\
\displaystyle &\hspace{0.3cm} \leq\|\nabla
u_0\|_{L^2}^{2}+\left(\sup_{0\leq s\leq t}\|\nabla_{h}
u\|_{L^2}\right)\left(\displaystyle\int_{0}^{t}\|\nabla
u\|_{L^2}^{2}d\tau \right)^{\frac{1}{4}}\\
&\ \ \ \ \ \times\left(\displaystyle\int_{0}^{t}\|\nabla_{h}\nabla
u\|_{L^2}^{2}d\tau\right)^{\frac{1}{2}}\left(\displaystyle\int_{0}^{t}\|\Delta
u\|_{L^2}^{2}d\tau\right)^{\frac{1}{4}}\\
\displaystyle &\hspace{0.3cm} \leq C\displaystyle\left(\int_0^t
 \left\|\|\partial_3u_3\|_{L^{\alpha}_{x_3}}\right\|_{L^{\beta}_{x_1,x_2}}
 ^{\frac{2\alpha\beta}{(s-1)\alpha\beta-(s-2)\beta-2\alpha}}
 \left\|u_3\right\|_{L^{s}}^{\frac{2(\alpha-1)\beta s}{(s-1)\alpha\beta-(s-2)\beta-2\alpha}}
 \|\nabla u
\|_{L^2}^{2}d\tau\right)\\
 &\hspace{0.5cm}\displaystyle \ \ \times
 \left(\int_{0}^{t}\|\Delta
 u\|_{L^2}^{2}d\tau\right)^{{\frac{1}{4}}}+ \displaystyle\|\nabla_{h}
u_0\|_{L^2}^{2}\left(\displaystyle\int_{0}^{t}\|\Delta
u\|_{L^2}^{2}d\tau\right)^{\frac{1}{4}}+\|\nabla u_0\|_{L^2}^{2}.
\end{array}\end{equation}
By H$\ddot{\mbox{o}}$lder's and Young's inequalities, one has
\begin{equation} \label{3.47}\begin{array}{ll}
&\|\nabla
u\|_{L^2}^{2}+\displaystyle2\nu\displaystyle\int_{0}^{t}\|\Delta
u\|_{L^2}^{2}d\tau\\
\displaystyle &\hspace{0.3cm} \leq C\displaystyle\left(\int_0^t
 \left\|\|\partial_3u_3\|_{L^{\alpha}_{x_3}}\right\|_{L^{\beta}_{x_1,x_2}}
 ^{\frac{2\alpha\beta}{(s-1)\alpha\beta-(s-2)\beta-2\alpha}}
 \left\|u_3\right\|_{L^{s}}^{\frac{2(\alpha-1)\beta s}{(s-1)\alpha\beta-(s-2)\beta-2\alpha}}
 \|\nabla u
\|_{L^2}^{2}d\tau\right)^{\frac{4}{3}}\\
 &\hspace{0.5cm} \ \ \displaystyle +C\left(1+\|\nabla
u_0\|_{L^2}^{\frac{8}{3}}\right)+\nu\int_{0}^{t}\|\Delta
 u\|_{L^2}^{2}d\tau\\
\displaystyle &\hspace{0.3cm} \leq C\displaystyle\left(\int_0^t
 \left\|\|\partial_3u_3\|_{L^{\alpha}_{x_3}}\right\|_{L^{\beta}_{x_1,x_2}}
 ^{\frac{8\alpha\beta}{3((s-1)\alpha\beta-(s-2)\beta-2\alpha)}}
 \left\|u_3\right\|_{L^{s}}^{\frac{8(\alpha-1)\beta s}{3((s-1)\alpha\beta-(s-2)\beta-2\alpha)}}
 \|\nabla u
\|_{L^2}^{2}d\tau\right)\\
 &\hspace{0.5cm} \ \ \displaystyle\times
 \left(\int_{0}^{t}\|\nabla
 u\|_{L^2}^{2}d\tau\right)^{{\frac{1}{4}}}+ \displaystyle C\left(1+\|\nabla
u_0\|_{L^2}^{\frac{8}{3}}\right)+\nu\int_{0}^{t}\|\Delta
 u\|_{L^2}^{2}d\tau.
\end{array}\end{equation}
Absorbing the last term on the right hand side of \eqref{3.31}, and
thanks to the energy inequality \eqref{3.1}, we get
\begin{equation} \label{3.48}\begin{array}{ll}
&\|\nabla
u\|_{L^2}^{2}+\displaystyle\nu\displaystyle\int_{0}^{t}\|\Delta
u\|_{L^2}^{2}d\tau\leq C\left(1+\|\nabla
u_0\|_{L^2}^{\frac{8}{3}}\right)\\
\displaystyle &\hspace{0.3cm} \displaystyle+ C\left(\int_0^t
 \left\|\|\partial_3u_3\|_{L^{\alpha}_{x_3}}\right\|_{L^{\beta}_{x_1,x_2}}
 ^{\frac{8\alpha\beta}{3((s-1)\alpha\beta-(s-2)\beta-2\alpha)}}
 \left\|u_3\right\|_{L^{s}}^{\frac{8(\alpha-1)\beta s}{3((s-1)\alpha\beta-(s-2)\beta-2\alpha)}}
 \|\nabla u
\|_{L^2}^{2}d\tau\right).\
\end{array}\end{equation}
Now, we set the pair of  conjugate indexes as follows$$
h=\frac{3[(s-1)\alpha\beta-(s-2)\beta-2\alpha]}{4(2\alpha\beta
-\beta-2\alpha)},
$$
and
$$
h^{\prime}=\frac{3[(s-1)\alpha\beta-(s-2)\beta-2\alpha]}{3\alpha\beta
s-11\alpha\beta-3\beta s+10\beta+2\alpha},
$$
where we note that  \eqref{z5} implies $h\geq1$ and then
$h^{\prime}>1.$ Therefore, by Young's inequality and \eqref{3.48},
it  follows that
\begin{equation} \label{3.49}\begin{array}{ll}
&\|\nabla
u\|_{L^2}^{2}+\displaystyle\nu\displaystyle\int_{0}^{t}\|\Delta
u\|_{L^2}^{2}d\tau\leq C\left(1+\|\nabla
u_0\|_{L^2}^{\frac{8}{3}}\right)\\
\displaystyle &\hspace{0.3cm} \displaystyle+ C\left(\int_0^t\left(
 \left\|\|\partial_3u_3\|_{L^{\alpha}_{x_3}}\right\|_{L^{\beta}_{x_1,x_2}}
 ^{\frac{2\alpha\beta}{2\alpha\beta-\beta-2\alpha}}+
 \left\|u_3\right\|_{L^{s}}^{\frac{8(\alpha-1)\beta s}{3\alpha\beta
s-11\alpha\beta-3\beta s+10\beta+2\alpha}}
 \right)\|\nabla u
\|_{L^2}^{2}d\tau\right).\
\end{array}\end{equation}
Applying \eqref{3.41a}, 
we have
\begin{equation} \label{3.52}\begin{array}{ll}
\displaystyle\frac{3}{s}+\frac{3\alpha\beta s-11\alpha\beta-3\beta
s+10\beta+2\alpha}{4(\alpha-1)\beta s} &\displaystyle= \frac{
\alpha\beta-2\beta+3 \alpha\beta s-3\beta s+2\alpha}{4(\alpha-1)\beta s}\\
&\displaystyle=1-\frac{\epsilon}{4}.
\end{array}\end{equation}
Therefore, since $\epsilon$ is arbitrary, by Gronwall's inequality
and \eqref{5a4}-\eqref{z5}, \eqref{3.49} implies that
\begin{equation} \label{3.53}\begin{array}{ll}
&\|\nabla
u\|_{L^2}^{2}+\displaystyle\nu\displaystyle\int_{0}^{t}\|\Delta
u\|_{L^2}^{2}d\tau\leq C\left(1+\|\nabla
u_0\|_{L^2}^{\frac{8}{3}}\right)(1+M)e^{CM},
\end{array}\end{equation}
 for all  $t\in(0,T^{*})$. Therefore, the $ H^1$ norm of the strong
solution $u$ is bounded on the maximal interval of existence $(0,
T^{*})$. This completes the proof of Theorem \ref{t1.4}.
\end{proof}

\begin{proof}[Proof
of Theorem \ref{t1.5}]Since $\alpha, \beta$ and $ s$ satisfy
\eqref{z7} and \eqref{z8}, for any arbitrary small positive constant
$\epsilon$ satisfying $0<\epsilon <\min\left\{\frac{2}{9},
\frac{2\beta-2}{\beta}\right\}$,
 we can choose $\alpha$ such that
\begin{equation}\label{3.54}
\frac{4\beta}{(8-\epsilon)\beta-8}\leq\alpha\leq
\frac{\beta}{(2-\epsilon)\beta-2},
\end{equation}
 and then we choose
\begin{equation}\label{3.55}
s=\frac{(1+\epsilon)\alpha\beta-2\beta+2\alpha}{(\alpha-1)\beta}.
\end{equation}
From \eqref{3.54}, it is easy to check that
$$
3\leq s\leq \frac{9\alpha\beta-6\beta-6\alpha}{(\alpha-1)\beta}.
$$
Next, we set $r, \theta, a, t$   as in
\eqref{3.41aa},\eqref{3.42},\eqref{3.43}, \eqref{3.44} respectively.
From \eqref{3.41aa}, we have
$$
r-2=\frac{(s-1)\alpha\beta-(s-2)\beta-2\alpha}{\alpha\beta+\alpha-\beta},
$$
and by \eqref{z8}, we have  $2\alpha\beta -\beta-2\alpha>0.$
Therefore, one has
$$
2\alpha\beta -\beta-2\alpha\geq0\Longleftrightarrow
\frac{\alpha\beta-2\beta+2\alpha}{(\alpha-1)\beta}<3\leq s
\Longleftrightarrow(s-1)\alpha\beta-(s-2)\beta-2\alpha>0,
$$
 and finally we get $r>2 $. Since
$$
(s-1)\alpha\beta>(s-2)\beta+2\alpha\Longleftrightarrow
s\alpha\beta-s\beta-\alpha+\beta>\alpha\beta+\alpha-\beta=(\alpha-1)\beta+\alpha>0,
$$
it is easy for us to get $0\leq\theta<1.$ Moreover,  $a\geq 1$ and
$t\geq 1$ are obviously.  All parameters selected above satisfy the
conditions of Lemma \ref{l2.2}.  Similar to the proof of Theorem
\ref{t1.4}, we have \eqref{3.48}, and we restate below
\begin{equation} \label{3.60}\begin{array}{ll}
&\|\nabla
u\|_{L^2}^{2}+\displaystyle\nu\displaystyle\int_{0}^{t}\|\Delta
u\|_{L^2}^{2}d\tau\leq C\left(1+\|\nabla
u_0\|_{L^2}^{\frac{8}{3}}\right)\\
\displaystyle &\hspace{0.3cm} \displaystyle+ C\left(\int_0^t
 \left\|\|\partial_3u_3\|_{L^{\alpha}_{x_3}}\right\|_{L^{\beta}_{x_1,x_2}}
 ^{\frac{8\alpha\beta}{3((s-1)\alpha\beta-(s-2)\beta-2\alpha)}}
 \left\|u_3\right\|_{L^{s}}^{\frac{8(\alpha-1)\beta s}{3((s-1)\alpha\beta-(s-2)\beta-2\alpha)}}
 \|\nabla u
\|_{L^2}^{2}d\tau\right).\
\end{array}\end{equation}
Now, we set the pair of  conjugate index as follows
$$
h=\frac{3[(s-1)\alpha\beta-(s-2)\beta-2\alpha]}{4\beta(\alpha-1)(s-3)},
$$
and
$$
h^{\prime}=\frac{3[(s-1)\alpha\beta-(s-2)\beta-2\alpha]}
{9\alpha\beta-\alpha\beta s+\beta s-6\beta-6\alpha},
$$where we note that  \eqref{z8} implies $h\geq1$ and then
$h^{\prime}>1.$ Therefore, by Young's inequality, \eqref{3.60}
implies that
\begin{equation} \label{3.61}\begin{array}{ll}
&\|\nabla
u\|_{L^2}^{2}+\displaystyle\nu\displaystyle\int_{0}^{t}\|\Delta
u\|_{L^2}^{2}d\tau\leq C\left(1+\|\nabla
u_0\|_{L^2}^{\frac{8}{3}}\right)\\
\displaystyle &\hspace{0.3cm} \displaystyle+ C\left(\int_0^t\left(
 \left\|\|\partial_3u_3\|_{L^{\alpha}_{x_3}}\right\|_{L^{\beta}_{x_1,x_2}}
 ^{\frac{8\alpha\beta}{9\alpha\beta-\alpha\beta s+\beta s-6\beta-6\alpha}}+
 \left\|u_3\right\|_{L^{s}}^{\frac{2s}{s-3}}
 \right)\|\nabla u
\|_{L^2}^{2}d\tau\right).\
\end{array}\end{equation}
Applying \eqref{3.55}, 
we have
\begin{equation} \label{3.62}\begin{array}{ll}
\displaystyle\frac{1}{\alpha}+\frac{2}{\beta}+\frac{9\alpha\beta-\alpha\beta
s+\beta s-6\beta-6\alpha}{4\alpha\beta} &\displaystyle=
\frac{9\alpha\beta+\beta
s-2\beta+2\alpha-\alpha\beta s}{4\alpha\beta}\\
&\displaystyle=2-\frac{\epsilon}{4}.
\end{array}\end{equation}
Therefore, since $\epsilon$ is arbitrary, by Gronwall's inequality
and \eqref{5a5}-\eqref{z8}, \eqref{3.61} implies that
\begin{equation} \label{3.33}\begin{array}{ll}
&\|\nabla
u\|_{L^2}^{2}+\displaystyle\nu\displaystyle\int_{0}^{t}\|\Delta
u\|_{L^2}^{2}d\tau\leq C\left(1+\|\nabla
u_0\|_{L^2}^{\frac{8}{3}}\right)(1+M)e^{CM},
\end{array}\end{equation}
 for all  $t\in(0,T^{*})$. Therefore, the $ H^1$ norm of the strong
solution $u$ is bounded on the maximal interval of existence $(0,
T^{*})$. This completes the proof of Theorem \ref{t1.5}.
\end{proof}

}
\end{document}